\documentclass{gtpart}     
\usepackage{amsmath,amssymb,amsfonts,amsthm,epsfig,latexsym,graphics,enumitem}
\usepackage[usenames,dvipsnames]{color}

  \def\cH{{\cal H}}

\def\cF{{\cal F}}

\title{Affine Hirsch foliations on $3$-manifolds}

\author{B. Yu}
\givenname{Bin}
\surname{Yu}
\address{School of Mathematical Sciences\\  Tongji University\\ 200092 Shanghai\\ CHINA}
\email{binyu1980@gmail.com}
\urladdr{}

\keyword{affine Hirsch foliation, classification, exchangeable braid}
\subject{primary}{msc2010}{57R32}
\subject{secondary}{msc2010}{57M25}

\arxivreference{arXiv:1604.03723}
\arxivpassword{}

\newtheorem{theorem}{{Theorem}}[section]
\newtheorem{proposition}[theorem]{{Proposition}}

\newtheorem{lemma}[theorem]{{Lemma}}

\newtheorem{corollary}[theorem]{{Corollary}}

\newtheorem{claim}[theorem]{{Claim}}
\newtheorem{question}[theorem]{{Question}}

\newtheorem{definition}[theorem]{{Definition}}

\theoremstyle{definition}
\newtheorem{remark}[theorem]{{Remark}}

\begin{document}

\begin{abstract}    
This paper is devoted to discussing affine Hirsch foliations on $3$-manifolds.
First, we prove that  up to isotopic leaf-conjugacy, every closed orientable $3$-manifold $M$ admits
$0$, $1$ or $2$ affine Hirsch foliations. Furthermore, every case is possible.

Then, we analyze the $3$-manifolds admitting two affine Hirsch foliations
(abbreviated as  Hirsch manifolds).
On the one hand, we construct Hirsch manifolds by using exchangeable braided links (abbreviated as DEBL Hirsch manifolds);
on the other hand, we show that every Hirsch manifold virtually  is a DEBL Hirsch manifold.

Finally, we show that for every $n\in \mathbb{N}$,  there are only finitely many  Hirsch manifolds with strand number $n$.
Here the strand number of a Hirsch manifold $M$ is a positive integer defined by using strand numbers of braids.
\end{abstract}

\maketitle


\section{Introduction}

In 1975, Hirsch \cite{Hir} constructed an analytic 2-foliation on a closed 3-manifold so that the foliation contains exactly
one exceptional minimal set. Let's briefly recall his construction here.

The foliation is constructed by starting with a solid torus and
removing from the interior another solid torus which
wraps  around the original solid torus twice. This gives us a manifold, foliated by
$2$-punctured disks, with two transverse
tori as boundary components. We then glue the exterior boundary component to the interior boundary
component to obtain a foliated manifold without boundary. Hirsch chose a gluing map carefully
so that the $2$-punctured fibration structure induces a foliation and
 the induced foliation is analytic and contains exactly
one exceptional minimal set.

There are many  variations of  Hirsch's construction in the literature, for instance, Ghys \cite{Gh}, Bis-Hurder-Shive \cite{BHS}.
\begin{itemize}
  \item  Ghys \cite{Gh} considered a variant of Hirsch's construction: change  Hirsch's gluing map to an `affine' map in some sense.
 In \cite{BHS}, the authors call this foliation affine Hirsch foliation.
  \item In \cite{BHS}, the authors generalize  Hirsch's construction a lot. In the $3$-dimensional case,
  they generalize Hirsch's construction by starting with a solid torus $V$ and
removing from $V$ a small solid torus $V_0$  so that $V_0$
can be regarded as a tubular neighborhood of a closed twisted braid in $V$.
\end{itemize}

Actually, it is natural to generalize these foliations to a more popular case by using braids:
\begin{itemize}
  \item For every $n$-braid
$b$ whose closure is a knot,
starting with a solid torus $V$ and removing from the interior a small solid torus $V_0$
 which is a small tubular neighborhood of the closure of $b$, we get a compact $3$-manifold, foliated by
 $n$-punctured disks, with two boundary components transverse to the $n$-punctured disk fibration.
  \item  Then we  glue the exterior boundary component to the  interior boundary
component to obtain a foliated manifold induced by the  $n$-punctured disk fibration.
\end{itemize}

To simplify, we still call the new foliations  Hirsch foliations, which are the main objects in this paper.
Similarly, if  the gluing map is `affine' in some sense, we call  the  Hirsch foliation \emph{affine}.
More precise definitions can be found in Section \ref{s.prel}.

There are several kinds of discussions about Hirsch foliations in the literature:
\begin{itemize}
  \item Bis-Hurder-Shive \cite{BHS} generalized Hirsch's construction to construct analytic
  foliations of arbitrary codimension with exotic minimal sets.
  \item Alvarez and Lessa \cite{AL} considered the Teichm\"{u}ller space of a Hirsch foliation.
  \item Shive  in his thesis  \cite{Sh} considered a  conjugacy  question:
  fixing two Hirsch foliations $(M_1,\cH_1)$ and $(M_2,\cH_2)$, a $C_r$ leaf-conjugacy  diffeomorphism $H:M_1 \to
M_2$ and an integer $k\in \mathbb{N}$, how to find conditions
on the foliations and the map $H$ which ensure that the map $H$ is  $C_{k+\lambda}$?
\end{itemize}

In this paper, we also would like to discuss a conjugacy question. Different to what Shive did, we hope to understand
the leaf-conjugacy classes of Hirsch foliations.
We say two foliations $\cH_1$ and $\cH_2$ on a closed $3$-manifold $M$ are \emph{isotopically leaf-conjugate} if there exists a homeomorphism $h:M\to M$ which maps every leaf of $\cH_1$ to a leaf of $\cH_2$ and is isotopic to the identity map on $M$. We say that $\cH_1$ and $\cH_2$ are the same
up to \emph{isotopic leaf-conjugacy} if $\cH_1$ and $\cH_2$ are isotopically leaf-conjugate.
In this paper, we will restrict ourselves to
affine Hirsch foliations. The reasons why we focus on affine Hirsch foliations are the following.
\begin{itemize}
  \item A Hirsch foliation always can be easily rebuilt (see Remark \ref{r.HirvsaffHir} )  by modifying  the  gluing map of an affine Hirsch foliation.
  \item Affine Hirsch foliations are  natural objects in dynamical systems:
   the projection of the stable manifolds of a Smale solenoid attractor on the orbit space of the wandering set (by a Smale solenoid mapping on
a solid torus) is an affine Hirsch foliation. A forthcoming paper \cite{Yu} will focus on this topic.
\end{itemize}

Now  we can naturally ask the following question as the main motivation for this paper.

\begin{question}\label{q.mainq}
Let $M$ be a closed $3$-manifold, can we classify all  affine Hirsch foliations up to isotopic  leaf-conjugacy?
\end{question}

Note that Alvarez and Lessa \cite [Section 1.3] {AL}  have discussed this question
on the $3$-manifolds constructed by Hirsch.
As a first step to answer Question \ref{q.mainq}, we have:

\begin{theorem}\label{t.maint}
Let $M$ be a closed orientable 3-manifold. Then
 $M$ admits  $0$, $1$, or $2$ affine Hirsch foliations up to isotopic leaf-conjugacy.
\end{theorem}

Then one naturally would like to know:
\begin{question}\label{q.mfds}
\begin{enumerate}
  \item Which $3$-manifolds
admit a Hirsch foliation?
  \item Which $3$-manifolds admit two non-isotopically leaf-conjugate affine Hirsch foliations and what are the relations between these two  foliations?
\end{enumerate}
\end{question}

Actually, to the first item of Question \ref{q.mfds},
on the one hand, these manifolds are very clear, i.e. everyone is precisely decided by a braid and a gluing map;
on the other hand, it is not easy to describe all of these manifolds in a familiar and comfortable way.
Nevertheless, we would like to give some characterizations about these $3$-manifolds.

\begin{proposition}\label{p.topchar}
Let $M$ be a closed orientable $3$-manifold which admits an (affine) Hirsch foliation, then:
\begin{enumerate}
  \item $M$ is a toroidal $3$-manifold whose JSJ diagram is cyclic;
  \item each JSJ piece is either hyperbolic or a $S(0,2;\frac{q}{p})$ type Seifert manifold where $p$ and $q$ ($0<q<p$) are coprime.
\end{enumerate}
\end{proposition}

This proposition is a consequence of Lemma \ref{l.topologyN} and Corollary \ref{c.incomp}.

We are more interested in the second item of Question \ref{q.mfds}.
We call a $3$-manifold $M$ a \emph{Hirsch manifold}
if $M$ admits two non-isotopically leaf-conjugate Hirsch foliations.
Notice that the $3$-manifold constructed by Hirsch
in \cite{Hir} actually  is a Hirsch manifold.
Actually, there are a lot of Hirsch manifolds, see subsection \ref{sb.His} and Proposition \ref{p.nonisotopy}.
The following are the reasons why we are interested in Hirsch manifolds:
\begin{itemize}
  \item a Hirsch manifold has some nice symmetric structures;
  \item Hirsch manifolds and their two affine Hirsch foliations will play a central role in a class of dynamical systems:
in \cite{Yu}, the author will use Hirsch manifolds and affine Hirsch foliations to discuss a kind of $\Omega$-stable diffeomorphisms on $3$-manifolds whose nonwandering set is the union of a Smale solenoid attractor and a Smale solenoid repeller.
\end{itemize}

Exchangeably braided links introduced by Morton  \cite{Mor} will play a crucial role to describe Hirsch manifolds.
An \emph{exchangeably braided  link} is a two-component link $L=K_1 \cup K_2$ in $S^3$
so that each component is braided relative to the other one.  More details about exchangeably braided  link can be found in
Section \ref{s.prel}.

Motivated by the second item of Question \ref{q.mfds}, we will give  two observations to describe the relationships
between exchangeably braided  links and  Hirsch manifolds.
The first observation is that for every exchangeably braided link $L=K_1 \cup K_2$, one can build a (unique) Hirsch manifold following
a series of standard combinatorial surgeries (see Section \ref{s.twohirsch}).  Such a Hirsch manifold is called a Hirsch manifold derived from exchangeably braided link (abbreviated as
 a \emph{DEBL Hirsch manifold}).
The second observation is that every Hirsch manifold virtually is a  DEBL Hirsch manifold.  More precisely,

\begin{theorem}\label{t.covering}
Let $M$ be a Hirsch manifold. Then there exists  a $q_2$-covering space of $M$, named by $\widetilde{M}$,
so that $\widetilde{M}$ is a Hirsch manifold derived from an exchangeably braided link
 (DEBL Hirsch manifold).
  Moreover, $q_2$ can be divided by $n^2-1$ where $n$ is the strand number of $M$.
\end{theorem}

Here, the \emph{strand number} of a Hirsch manifold $M$ (see Definition \ref{d.strandnumber}) is defined to be the strand number of a braid
which can be used to build the Hirsch manifold $M$.

Hirsch manifolds have the following finiteness property.
\begin{proposition}\label{p.finiteMs}
For every $n\in \mathbb{N}$, there are only finitely many Hirsch manifolds
with strand number $n$.
\end{proposition}

In the final section (Section \ref{s.example}), we will build an example to show:

\begin{proposition}
\label{p.notHirschmfd}
There exists a $3$-manifold which admits a Hirsch foliation but is not a Hirsch manifold.
\end{proposition}

Proposition \ref{p.topchar}, Proposition \ref{p.nonisotopy}, the examples in Section \ref{s.twohirsch}
and Proposition \ref{p.notHirschmfd} imply that
there exist  closed oriented  $3$-manifolds $M_0$, $M_1$ and $M_2$ so that,
\begin{itemize}
  \item $M_0$ doesn't admit any affine Hirsch foliation;
  \item $M_1$ admits exactly one affine Hirsch foliation;
  \item $M_2$ is a Hirsch manifold, i.e. $M_2$ admits two non-isotopically leaf-conjugate affine Hirsch foliations.
\end{itemize}
This means that every case in Theorem \ref{t.maint} can be realized (in some closed $3$-manifold).

We can see that the results (in particular, Proposition \ref{p.nonisotopy} and Theorem \ref{t.maint}) in this paper give a satisfying response to classify all Hirsch foliations.
After these results, we can reduce classifying all Hirsch foliations to a  classical problem in one dimensional dynamical systems:
classifying degree $n$ ($n\geq 2$) endomorphisms \footnote{An endomorphism on $S^1$ means
a monotonic continuous map on $S^1$} on $S^1$ up to conjugacy. More details can be found in Remark \ref{r.HirvsaffHir}.

\section{Preliminaries}
\label{s.prel}

\begin{definition}\label{d.hirsch}
Let $\cH$ be a codimension $1$ foliation on a   closed oriented $3$-manifold $M$.
$\cH$ is called a \emph{Hirsch foliation}
if there exists a torus $T$ embedded into $M$ so that,
\begin{enumerate}
  \item the path closure of $M-T$, $N$, is a compact oriented $3$-manifold with two tori $T^{out}$ and $T^{in}$ as the boundary;
  \item $\cH\mid_N$ is an $n$-punctured disk fibration on $N$ such that each fiber is transverse to $\partial N$;
  \item every leaf in $\cH$ is orientable.
\end{enumerate}
\end{definition}

By Definition \ref{d.hirsch}, a Hirsch foliation $\cH$ on a closed oriented $3$-manifold $M$  can be constructed as follows
(see Figure \ref{f.defhirschfoliation} for an illustration).

 \begin{figure}[htp]
\begin{center}
  \includegraphics[totalheight=5.5cm]{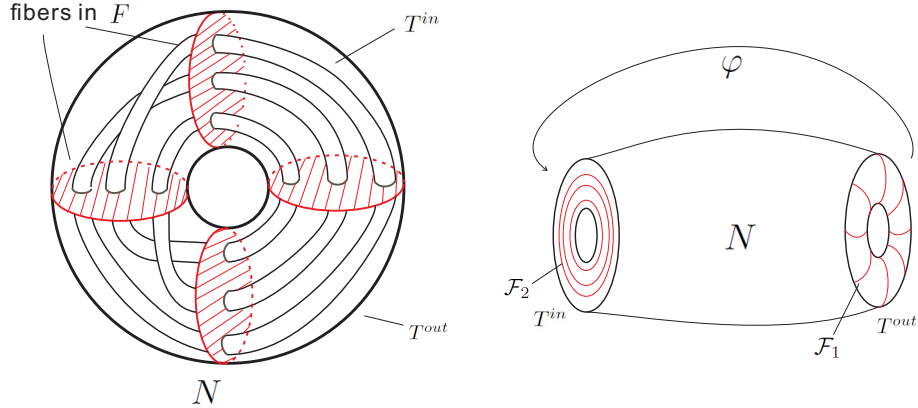}
  \caption{an example about the constructions of Hirsch foliations}\label{f.defhirschfoliation}
  \end{center}
\end{figure}

\begin{itemize}
  \item Choose an $n$-braid $b$ whose closure is a knot, $b$ also can be used to represent a diffeomorphism
  on an $n$-punctured disc $\Sigma$.
  \item We denote the mapping torus of $(\Sigma, b)$ by $N$. Notice that $F= \{\Sigma \times \{\star\}\}$ provides a natural
$n$-punctured disk fibration on  $N$, which provides $T^{in}$ and $T^{out}$ two $S^1$-fibration structures $\cF_1$ and $\cF_2$
respectively.
  \item We suppose that $\Sigma$ is oriented.
Then $\Sigma$ naturally induces an orientation on
each fiber of $\cF_1$ and $\cF_2$.
  We also give an orientation on $N$ which naturally induces two orientations on $T^{out}$ and $T^{in}$ respectively.
  \item Build an orientation-preserving homeomorphism $\varphi: T^{out} \to T^{in}$ which maps every fiber of $\cF_1$ to a fiber of $\cF_2$
  and preserves the corresponding orientations\footnote{$\varphi:T^{out} \to T^{in}$ preserves the orientations since the glued manifold $M$ should be orientable. $\varphi$ preserves the corresponding orientations of the fibers of $\cF_1$ and $\cF_2$ since every leaf in the glued foliation $\cH$ should be orientable}.
  Let $M =N\setminus x\sim \varphi(x), x\in T^{out} N$.
  Then the  $n$-punctured disk fibration $F$ on $N$ naturally induces a Hirsch foliation $\cH$ on $M$ by $\varphi$.
  \end{itemize}

There are some further comments about Hirsch foliations which are useful in the rest of the paper:
\begin{itemize}
  \item $N$ also can be  obtained  by   removing  a small solid torus $V_0$ from a solid torus $V$ where
$V_0$ is a small tubular neighborhood of the closure of  a braid $b$;
  \item there is a natural quotient map $P:N \to S^1$ where $S^1$ is the fiber quotient space of $F$;
  \item $\varphi$ induces a  map $\varphi_2: S^1 \to S^1$, which is called the \emph{projective holonomy map} of $\cF$
  relative to the embedded torus $T$.
\end{itemize}

\begin{definition}\label{d.affinehirsch}
Let $\cH$ be a Hirsch foliation on a   closed $3$-manifold $M$.
$\cH$ is called an \emph{ affine Hirsch foliation}
if the projective holonomy map of $\cF$
relative to an embedded torus $T$ transverse to $\cH$ is topologically conjugate to the map $z^n$ on $S^1$ for some
$n \in \mathbb{N}$, $n\geq 2$. Here we can parameterize $S^1$ by $S^1= \{z\mid |z|=1, z\in \mathbb{C}\}$.
\end{definition}

In 1985, Morton \cite{Mor} introduced exchangeably braided links.
 An \emph{exchangeably braided  link} is a two-component link $L=K_1 \cup K_2$  which admits a kind of very nice symmetry:  each component is braided relatively to the other one, \emph{i.e.} $K_1$ is a closed braid $\widetilde{b_1}$ in the solid torus $S^3 - K_2$ and
 $K_2$ is a closed braid $\widetilde{b_2}$ in the solid torus $S^3 -K_1$.  Such a braid $b_1$ is called an \emph{exchangeable braid}.
 Automatically, every exchangeably braided link $L$ can be regarded as the union of the closure of an exchangeable braid and
an axis of the closed braid.

Morton \cite{Mor}  showed many nice properties of exchangeably braided  links. For instance, he built
some necessary and some sufficient conditions for exchangeability. For instance, he showed that the exchangeable braids  belong to a family of braids introduced by Stallings \cite{St}.

Let's briefly introduce Stallings' braids  and the relationships between Stallings' braids and exchangeable braids.
Certainly, the closure of an exchangeable braid is a trivial knot. But the converse is not true, \emph{i.e.}
if the closure of a braid $b$ is a trivial knot, $b$ is not necessarily  an exchangeable braid.
Actually, Stallings \cite{St} introduced a family of braids so that for every braid $b$ in this family satisfies:
\begin{enumerate}
  \item $\widetilde{b}$ is a trivial knot;
  \item there is a disk $D$ spanning $\widetilde{b}$
  which intersects the axis at exactly $n$-points.
\end{enumerate}
Morton called them \emph{Stallings braids}.
The set of Stallings braids is a proper subset of the union of the braids whose closure is a trivial knot.
In \cite{Mor}, Morton constructed a braid $\omega=\sigma_3 \sigma_2 \sigma_3^{-1} \sigma_2 \sigma_1^{-1} \sigma_2 \sigma_1 \in B_4$
\footnote{Here and below, the notations for braids are standard in braid theory
(see, for instance, Birman \cite{Bi})}
which is a stalling  braid but not an exchangeable braid.
Therefore, the union of the exchangeable braids is a proper subset of
the union of  Stallings braids.

Stallings braids have a very nice characterization \footnote{A careful reader can find the characterization in the beginning of  \cite [Section $2$] {Mor}}.
Under this characterization, it is easy to obtain the following finiteness property.

\begin{proposition}\label{p.finitestallings}
For a given $n\in \mathbb{N}$, up to conjugacy, there are finitely many Stallings braids with strands number $n$.
\end{proposition}

The following corollary is an immediate consequence of this proposition (this corollary exactly is  \cite [Corollary 1.2] {Mor}).

\begin{corollary}\label{c.finiteness}
\begin{enumerate}
  \item Up to  conjugacy,
there are finitely many exchangeable braids with strand number $n$.
  \item Up to isotopy, there are finitely many exchangeably braided links
with linking number $n$.
\end{enumerate}
 \end{corollary}

\section{Proof of Theorem \ref{t.maint}}
\label{s.proof}

This section is devoted to proving Theorem  \ref{t.maint}. We will prove an equivalent form of the theorem:
if a closed $3$-manifold $M$ admits a Hirsch foliation $\cF$, then  up to isotopic leaf-conjugacy,
$M$ admits at most two affine Hirsch foliations.

 \begin{figure}[htp]
\begin{center}
  \includegraphics[totalheight=7cm]{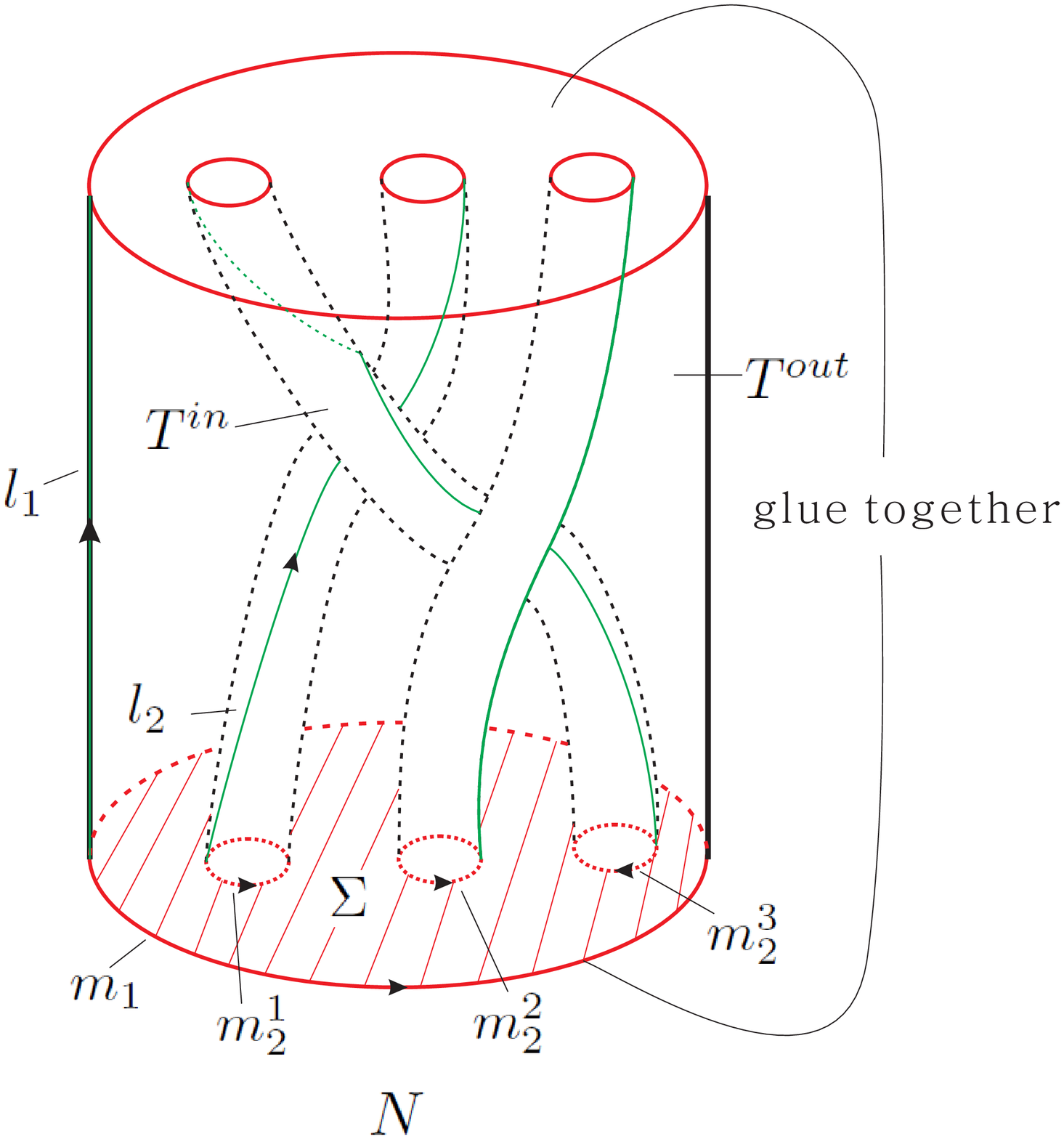}
  \caption{some notations in the proof of Theorem \ref{t.maint}}\label{f.notationproofthm1}
  \end{center}
\end{figure}

First, we give more notations and  parameters (see Figure \ref{f.notationproofthm1} as an illustration).

\begin{itemize}
  \item Assume that $i_1: T^{out} \to N$ and $i_2: T^{in} \to N$ are the associated embedding maps,
   $i_{1,\star}:H_1(T^{out}) \to H_1(N)$ and $i_{2,\star}:H_1(T^{in}) \to H_1(N)$ are the corresponding induced homomorphisms.
  \item We denote the oriented simple closed curve $\Sigma \cap T^{out}$ by $m_1$ and denote  the  oriented simple closed curves
  $\Sigma\cap T^{in}$ by $m_2^1,\dots, m_2^n$.
 Here, the orientations of the simple closed curves are induced by $\Sigma$.
  Sometimes, we also use $m_2$ to represent $m_2^1$.
  \item $l_1$ is chosen to be an oriented simple closed curve in $T^{out}$ which intersects  $m_1$ at one point so that
  $(m_1,l_1)$ endows $T^{out}$ an orientation which is coherent to the tori of $T^{out}$.
  \item $l_2$ is chosen to be the unique (up to isotopy in $T^{in}$) oriented simple closed curve in $T^{in}$ so that,
        \begin{enumerate}
          \item $l_2$ intersects $m_2$ at one point;
          \item $(m_2,l_2)$ endows $T^{in}$ an orientation  which is coherent to the orientation of $T^{in}$;
          \item $i_{2,\star}([l_2])= n\cdot i_{1,\star} ([l_1])$.
        \end{enumerate}
   \item If there are two oriented simple closed curves $m$ and $l$ on a torus $T^2$ which intersect at one point,  we will use
   $p m +q l$  ($p$ and $q$ are coprime) to represent an oriented simple closed curve on $T^2$ which
   wraps $p$ times around $m$ and $q$ times around $l$.
\end{itemize}

The existence and the uniqueness of $l_2$ can be shown by a short computation on homology, as follows.
If we choose a simple closed curve $l_2' \in T^{in}$ so that $l_2'$  intersects $m_2$ at one point and the orientation given by $(m_2, l_2')$
is coherent to the orientation of $T^{in}$ given by $N$,
then $i_{2,\star} ([l_2'])=n\cdot i_{1,\star} ([l_1])+x\cdot i_{2,\star} ([m_2])$ for some $x\in \mathbb{Z}$.
Since $i_{2,\star} ([m_2])$ is nonzero in $H_1 (N)$, then up to isotopy,
there is a unique simple closed curve $l_2 = l_2' - x\cdot m_2$ in $T^{in}$ so that   $i_{2,\star}([l_2])= n\cdot i_{1,\star} ([l_1])$.
To simplify, we will use $[m_j]$ and $[l_j]$ $(j=1,2)$ to represent the corresponding elements in $H_1 (N)$.

We collect  some information  about $H_1(N)$ as follows, which can be obtained  by Alexander duality.

\begin{lemma}
\label{l.homologyN}
$H_1 (N)\cong \mathbb{Z} \oplus \mathbb{Z}$ and is generated by $[l_1]$ and $[m_2]$.
Moreover,
$[m_1]=n[m_2]$ and $[l_2]=n[l_1]$.
\end{lemma}

The following lemma shows that the set of all punctured disk fibrations on $N$ is quite limited.
\begin{lemma}\label{l.sn}
Let $F$ be a $s$-punctured disk fibration on $N$. Assume that $\Sigma$ is a fiber of $F$ whose
boundary is the union of a simple closed curve $c_1 \in T^{out}$ and $s$ pairwise parallel and pairwise disjoint  simple closed
curves $c_2^1, \dots, c_2^s$ in $T^{in}$ (sometimes we  also use $c_2$ to represent $c_2^1$).
Assume that $c_i =p_i m_i + q_i l_i$ ($i=1,2$)
where $p_i$ and $q_i$  are coprime.
Then there exists an orientation on $\Sigma$ which induces an orientation on $c_1$ and an orientation on $c_2$
so that
 $s=n$, $p_1 =p_2 =1$ and $q_1=n^2 q_2$.
\end{lemma}
\begin{proof}
First let us prove that $p_2 =1$.
If we glue a solid torus $V$ to $N$ by a gluing map $\psi: \partial V \to T^{in}$ so that $c_2$
bounds a disk in $V$, then the glued $3$-manifold $U$ is homeomorphic to a solid torus.
On the one hand, it is obvious that $H_1 (U) = \langle[l_1]\rangle\cong \mathbb{Z}$.
On the other hand, $H_1(U)= \langle[m_2], [l_1]\mid p_2 [m_2]+ q_2 [l_2]=p_2 [m_2]+ nq_2[l_1]=0\rangle$.
Therefore, $p_2 =\pm1$.

Secondly, we can  endow $\Sigma$ an orientation which induces two orientations on $c_1$ and $c_2$ respectively
so that $p_2=1$. These orientations satisfy the requirements in the lemma and will be used in the following of the proof.

In the end,
we will prove that $p_1 = 1$ and $s = n$. Since the union of $c_2^1, \dots, c_2^s$ and $c_1$
bound a $s$-punctured disk $\Sigma$, $[c_1]=s[c_2]$. Equivalently, $p_1 [m_1]+q_1 [l_1]= s([m_2]+nq_2[l_1])$,
then, $q_1=sq_2 n$, $p_1 n=s$. We have $q_1= p_1 q_2 n^2$. Recall that $p_1$ and $q_1$ are coprime, therefore,
$p_1=1$, $s=n$ and $q_1= n^2 q_2$.
\end{proof}
\begin{remark}
Actually, for every $q_2 \in \mathbb{Z}$, there always exists an associated punctured disk fibration $F$ on $N$.
One can construct them by a standard surgery in low dimensional topology (for the surgery, see for instance,  Jaco \cite [III.14] {Ja}).
\end{remark}

From now on, $c_1$, $c_2$ and $\Sigma$ are oriented as Lemma \ref{l.sn}.

\begin{lemma}
\label{l.fixonT}
Let $\varphi:T^{out} \to T^{in}$ be a diffeomorphism so that $\varphi (m_1)= m_2$ and $\varphi (l_1)= l_2+ ym_2$
($y\in \mathbb{Z}$). If $\varphi(c_1)$ is isotopic to $c_2$ in $T^{in}$, then $c_1$ is isotopic to $m_1$ in $T^{out}$
and $c_2$ is isotopic to $m_2$ in $T^{in}$.
\end{lemma}

\begin{proof} On the one hand,
$\varphi_{\star} ([c_1])=\varphi_{\star}([m_1]+q_1 [l_1])=\varphi_{\star}([m_1])+q_1 \varphi_{\star} ([l_1])
=m_2 +q_1 (l_2+ ym_2)=(1+q_1 y)m_2+q_1 l_2$;
on the other hand, $\varphi_{\star} ([c_1]) =[c_2]=[m_2]+q_2 [l_2]$.
Therefore, $1+q_1 y=1$ and $q_1=q_2$. By Lemma \ref{l.sn}, $q_1=n^2 q_2$ and $|n|\geq 2$.
Then $q_1=q_2=0$.
Also notice that $p_1 =p_2 =1$ (Lemma \ref{l.sn}), the conclusions of the lemma follow.
\end{proof}

\begin{lemma}
\label{l.fixinterior}
If $F_1$ and $F_2$ are two $n$-punctured disk fibrations on $N$ with two fibers $\Sigma_1$ and $\Sigma_2$ correspondingly
so that $\partial \Sigma_1= \partial \Sigma_2$ is the union of $m_1$ and $n$ simple closed curves $m_2^1, \dots, m_2^n$
which are pairwise isotopic, then $\Sigma_1$ is isotopic to $\Sigma_2$ relative to $\partial \Sigma_1= \partial \Sigma_2$ in $N$.
\end{lemma}

\begin{proof}
Up to isotopy, we can assume that $\hbox{int}(\Sigma_1)\cap \hbox{int}(\Sigma_2)$ is the union of finitely many pairwise disjoint simple closed curves
$\alpha_1, \dots, \alpha_m$. Here $\hbox{int}(\Sigma_i)$ ($i=1,2$) is defined to be the interior of $\Sigma_i$.
 Moreover, we assume that $m\geq 1$  and $m$ is minimal   up to isotopy
relative to $\partial \Sigma_1= \partial \Sigma_2$.

Firstly, we will show that every $\alpha_i$ ($i\in \{1,\dots,m\}$) is essential in $\Sigma_2$. Otherwise, some $\alpha_i$ bounds a disk $D_2$ in
$\Sigma_2$. Notice that $\Sigma_1$ is incompressible in $N$, the union of $D_1$ and $D_2$, denoted by $S$, is a $2$-sphere embedded in $N$.
Since $N$ is an irreducible $3$-manifold, then $S$ bounds a $3$-ball in $N$. This means that we can do a surgery on $\Sigma_2$ in a small neighborhood of the $3$-ball to
obtain $\Sigma_2'$ so that $\Sigma_2'$ is isotopic to $\Sigma_2$ and the number of the connected components of
$\Sigma_2' \cap \Sigma_1$ is smaller than $m$. This contradicts the assumption.

Then there exists a nested $k$-punctured disk $D_1^k\subset \Sigma_1$ with boundary  $\alpha_j \cup (m_2^{s_1} \cup \dots \cup m_2^{s_k})$ for
some $j\in \{1, \dots, m\}$
where $\alpha_k$ is an essential simple closed curve in the interior of $\Sigma_1$. Here the fact that $D_1^k$ is a nested disk  means that the interior of $D_1^k$ is disjoint from $\Sigma_2$. We cut $N$ along $\Sigma_1$ to obtain a $3$-manifold $N_0$ which is homeomorphic to $\Sigma_1 \times [0,1]$.
Since $\partial \Sigma_1= \partial \Sigma_2$ and the fact that $D_1^k$ is a nested $k$-punctured disk, by a simple  argument on $N_0$, one can obtain that $\partial D_1^k$ also bounds a
nested $k$-punctured disk $D_2^k$ in $N_0$.
We define $(\Sigma_1- D_1^k) \cup D_2^k$ by $\Sigma_3$ which is an incompressible $k$-punctured  disk.
Since $N_0$ is homeomorphic to $\Sigma_1 \times [0,1]$, then $\Sigma_3$ is isotopic to $\Sigma_1$ relative to
$\overline{\Sigma_1 -D_1^k}$ in $N_0$. We can push $\Sigma_3$ a little into the interior of $N_0$ to $\Sigma_3'$ so that
the intersection number of $\Sigma_3'$ and $\Sigma_2$ is strictly smaller than the intersection number of
$\Sigma_1$ and $\Sigma_2$. This contradicts the minimality.
\end{proof}

Now we deal with the trouble that maybe there are  many incompressible tori in a Hirsch manifold. For this purpose,
we should read more topological information about $N$.

First, we recall some classical facts about the geometry and topology of surface bundles.
Nielsen-Thurston theorem (see, for instance, Fathi-Laudenbach-Poenaru \cite{FLP}) states that a homeomorphism $f$ on a compact surface $\Sigma$
is isotopic to one of the three types according to their dynamics: periodic, reducible
and pseudo-Anosov. Thurston geometrization theorem of surface bundles (see Thurston \cite{Th}) implies that  Nielsen-Thurston theorem
 deeply involves the geometric
structure of three dimensional manifolds as follows: the mapping torus $M_f = \Sigma \times I / (s,1)\sim (f(s),0)$ is an irreducible $3$-manifold, moreover,
\begin{enumerate}
  \item $M_f$ is hyperbolic if and only if $f$ is pseudo-Anosov;
  \item $M_f$ is Seifert-fibered if and only if $f$ is periodic;
  \item $M_f$ contains an essential torus
(hence we can perform JSJ decomposition) if and only if $f$ is reducible.
\end{enumerate}
In particular, in the third case, there exists a collection of essential simple closed curves in $\Sigma$
so that the suspension of these curves can be glued up by a map isotopic to $f$ to give a collection of
essential tori and Klein bottles which are the collection of JSJ tori and Klein bottles.
Now we formalize some facts about  the geometry and topology of surface bundles as follows which will be very useful.

\begin{lemma}\label{l.geotopsurfacebundle}
Let $M_f = \Sigma \times I / (s,1)\sim (f(s),0)$ be a mapping torus  where $\Sigma$ is a compact orientable surface and $f$ is an orientation-preserving homeomorphism on $\Sigma$.
Then,
\begin{enumerate}
  \item $M_f$ is an irreducible $3$-manifold and every JSJ piece of $M_f$ is either hyperbolic or Seifert;
  \item every JSJ torus $T$ of $M_f$ is corresponding to an essential simple closed curve $c$ in $\Sigma$ which is periodic up to
isotopy under $f$.
\end{enumerate}
\end{lemma}

Now, we come back to read some topological information about $N$.
\begin{lemma}\label{l.topologyN}
$N$ is an irreducible $3$-manifold so that,
\begin{enumerate}
 \item every JSJ piece of $N$ is either hyperbolic or Seifert;
  \item the JSJ diagram of $N$ is a  path;

  \item every Seifert piece is homeomorphic to $S(0,2;\frac{q}{p})$ where $p$ and $q$ ($0<q<p$) are coprime
  and $S(0,2;\frac{q}{p})$ represents the Seifert manifold whose base orbifold is a $2$-punctured  sphere with
  a $\frac{q}{p}$-singularity.
\end{enumerate}
\end{lemma}
\begin{proof}
Recall that $N$ can be defined to be the mapping torus of $(\Sigma, b)$
 where $\Sigma$ is an $n$-punctured compact
disk and $b$ is a homeomorphism on $\Sigma$.
Then by item $1$ of Lemma \ref{l.geotopsurfacebundle}, $N$ is an irreducible $3$-manifold so that every JSJ piece is either
hyperbolic or Seifert.

By item $2$ of Lemma \ref{l.geotopsurfacebundle}, every JSJ torus $T$ of $N$ is corresponding to an essential simple closed curve $c$ in $\Sigma$ which is periodic up to
isotopy under $b$.
On the one side, notice that every simple closed curve in $\Sigma$ is separating, then every JSJ torus of $N$
is separating. This implies that the JSJ diagram of $N$ is a tree.
On the other side, $\partial N$ is the union of two tori, $T^{out}$ and $T^{in}$.
Combing the two sides above, one could immediately obtain that the JSJ diagram of $N$ is a path.

Let $N_0$ be a Seifert piece of $N$, then $N_0$ is homeomorphic to a solid torus minus a small
open tubular neighborhood of a closed braid $\widetilde{b_0}$.
Since $N_0$ is Seifert, $b_0$ should be a periodic braid.
Since every periodic homeomorphism on a disk is conjugate to a rotation (see Constantin-Kolev \cite{CK}), up to conjugacy, $b_0$ should be a twisted braid.
This implies that $N_0$ is homeomorphic to a Seifert manifold $S(0,2; \frac{q}{p})$.

\end{proof}

Recall that  $M =N\setminus x\sim \varphi(x), x\in T^{out} N$. By Lemma \ref{l.topologyN}, the gluing map $\varphi$ glues the
two JSJ pieces corresponding to the two ends of the JSJ diagram of $N$ (notice that the two JSJ
pieces maybe are the same) and the two JSJ pieces should belong to one of the following three cases:
\begin{enumerate}
  \item both of them are hyperbolic;
  \item one of them is hyperbolic and the other one is Seifert;
  \item both of them are Seifert.
\end{enumerate}
In the first two cases, it is obvious that the glued torus $T$ is a JSJ torus in $M$.
In the third case, since $\varphi(m_1)=m_2$, one can easily check that up to isotopy, $\varphi:T^{out}\to
T^{in}$ doesn't  map a regular fiber on $T^{out}$ to a
regular fiber on $T^{in}$ (up to isotopy) induced by the associated Seifert  pieces. Therefore, $T$ is also a JSJ torus in $M$.
Now naturally we have the following corollary.

\begin{corollary}\label{c.incomp}
Let $M$ be a closed orientable $3$-manifold which admits a Hirsch foliation.
Then every incompressible torus $T$ embedded in $M$ is a JSJ torus and
the JSJ diagram of $M$ is cyclic.
\end{corollary}

\begin{lemma}
\label{l.transT}
Let $M$ be a closed $3$-manifold which admits an affine Hirsch foliation. We have the following conclusions.
\begin{enumerate}
  \item $M$ is the union of $n$ JSJ pieces $M_1, M_2, \dots, M_n$ by the gluing maps
  $\varphi_1: \partial^{out}M_1 \to \partial^{in} M_2, \dots, \varphi_{n-1}:  \partial^{out}M_{n-1} \to \partial^{in} M_n $
  and $\varphi_n: \partial^{out}M_n \to \partial^{in} M_1$. Here the
  union of $T_i^{out}$ and  $T_i^{in}$ is the boundary of $M_i$ ($i\in\{1,\dots,n\}$).
  \item Let $\{T_1, \dots, T_n\}$ be a union of the maximal pairwise disjoint, pairwise non-parallel JSJ tori of $M$ and
  $\cH$ be a Hirsch foliation on $M$. Then, $\cH$ can be isotopically leaf-conjugate to $\cH'$ so that every $T_i$
  ($i\in\{1,\dots,n\}$) is transverse to $\cH'$.
\end{enumerate}
\end{lemma}

\begin{proof}
Item $1$ of the lemma is a direct consequence of Corollary \ref{c.incomp}.
We only need to prove item $2$ of the theorem.

Without loss of generality, we can suppose that $T_i =\partial^{out} M_i$ ($i\in \{1,\dots,n\}$)
and  $\cH$ is transverse to $T_n$.
Let $N$ be the union of $M_1, M_2, \dots, M_n$ by the gluing maps $\varphi_1, \dots, \varphi_{n-1}$.
$\cH$ restricted to $N$ is an $m$-punctured disk fibration, denoted by $F$. Since $N$ admits an $m$-punctured
disk fibration  $F$, by Corollary \ref{c.incomp},
every incompressible torus $T$ in the interior of $N$ is a JSJ torus. Moreover, by item $2$ of Lemma \ref{l.geotopsurfacebundle},
 $T$ can be isotopic to $T'$ relative to $\partial N$
so that $T'$ is transverse to $F$.
Then by an easy inductive argument, $T_1, \dots, T_{n-1}$ in $N$ can be isotopic to $T_1', T_2', \dots, T_{n-1}'$ relative to
$\partial N$ respectively so that every $T_i'$ is transverse to $F$.
Equivalently, we can pertubate  $F$ in $N$ relative to $\partial N$ to $F'$ which is transverse to every $T_i$.
Then $F'$ naturally induces a foliation $\cH'$ in $M$ so that,
\begin{itemize}
  \item $\cH'$ is isotopically leaf-conjugate to $\cH$;
  \item $\cH'$ is transverse to every $T_i$.
\end{itemize}
\end{proof}

\begin{lemma}
\label{l.stdnber}
Let $M$ be a closed $3$-manifold which admits an affine Hirsch foliation $\cH$.
Let $T_1$ and $T_2$ be two incompressible tori in $M$ so that each of them is transverse to $\cH$.
We denote the path closure of $M-T_i$ by $N_i$ ($i=1,2$)
and denote the restriction of $\cH$ on $N_i$ by $\cF_i$ which is an $n_i$-punctured disk fibration on $N_i$.
Then $n_1= n_2$.
\end{lemma}

\begin{proof}
Without loss of generality, we can suppose that $T_1$ and $T_2$ are disjoint and non-parallel.
The path closure of $M-T_1\cup T_2$ is the union of two compact $3$-manifolds $W_1$ and $W_2$.
Actually, $N_1=W_1\cup_{T_2} W_2$ and $N_2 =W_2\cup_{T_1}W_1$. we denote $\cH$ restricted to $W_i$ ($i=1,2$)
by $H_i$ which is an $m_i$-punctured disk fibration on $W_i$.
Notice that every fiber of $\cF_1$ is the union of one fiber of $H_1$ and $m_1$ fibers of $H_2$.
Therefore, every fiber of $\cF_1$ is an $m_1\cdot m_2$-punctured disk. Equivalently, $n_1 = m_1\cdot m_2 $.
Similarly, $n_2 =m_2 \cdot m_1$.
In summary, $n_1 =n_2$.
\end{proof}

\begin{definition}
\label{d.strdnumberF}
Let $M$ be a closed $3$-manifold which admits an affine Hirsch foliation $\cF$ and
$T$ be an incompressible torus which is transverse to $\cF$. We denote $N$ by
the path closure of $M-T$ and denote $\cF$ restricted on  $N$ by $F$ which is   an $n$-punctured disk fibration.
We call $n$ the \emph{strand number} of $\cF$.
\end{definition}

\begin{remark}
Lemma \ref{l.transT} and Lemma \ref{l.stdnber} imply that the  the strand number of $\cF$
doesn't depend on the choice of $T$. Furthermore, by Lemma \ref{l.sn}, the stand number of an affine Hirsch foliation
is invariant under isotopic leaf-conjugacy.
\end{remark}

The following lemma explains that ``the affine property" of an affine foliation is independent of the choices of
$T$ and the foliations which are isotopically leaf-conjugate to the original affine foliation.

\begin{lemma}
\label{l.affT}
Let $M$ be a closed $3$-manifold which admits an affine Hirsch foliation $\cH_1$.
Let $\cH_2$ be a Hirsch foliation so that,
\begin{itemize}
  \item $\cH_2$ is isotopically leaf-conjugate to $\cH_1$;
  \item $\cH_2$ is transverse to an incompressible torus $T$ in $M$ and $N$ is the path closure of $M-T$.
\end{itemize}
Let $\cF_2$ be the punctured disk fibration on $N$ and $F_2$ be the circle fibration of $\cH_2$ restricting to $T$.
 Then, the projective holonomy map of $\cF$
relative to an embedded torus $T$ transverse to $\cH$ is topologically conjugate to the map $z^n$ on $S^1$
where $n$ is the strand number of $\cH_1$ and $\cH_2$.
\end{lemma}

To show Lemma \ref{l.affT}, by item $2$ of Lemma \ref{l.transT}, we only need to prove the following claim.

\begin{claim}
Let $M$ be a closed $3$-manifold which admits an affine Hirsch foliation $\cH$.
Let $T_1$ and $T_2$ be two incompressible tori in $M$.
Let $N_i$ ($i=1,2$) be the path closure of $M-T_i$, $\cF_i$
be the punctured disk fibration on $N_i$ and $F_i$ be the circle fibration of $\cH$ restricting to $T_i$.
Suppose that the projective holonomy map of $\cF_1$ relative to $T_1$, $\varphi_2^1$, is topologically conjugate
to the map $z^n$ on $S^1$. Then, the projective holonomy map of $\cF_2$ relative to $T_2$, $\varphi_2^2$,
 is also topologically conjugate to the map $z^n$ on $S^1$.
\end{claim}

\begin{proof}
By Lemma \ref{l.transT}, we can suppose that $T_1$ and $T_2$ are disjoint and non-parallel.
Let the path closure of $M-T_1 \cup T_2$ be the union of two compact $3$-manifolds $M_1$ and $M_2$ such that,
\begin{enumerate}
  \item $\partial M_i =\partial^{out} M_i \cup \partial^{in} M_i$ ($i=1,2$);
  \item $M$ is the union of $M_1$ and $M_2$ by the gluing maps $\varphi^1 : \partial^{out} M_1 \to \partial^{in} M_2$
  and $\varphi^2 : \partial^{out} M_2 \to \partial^{in} M_1$;
  \item $\partial^{out} M_1$ and $\partial^{in} M_2$ are corresponding to $T_1$ and $\partial^{out} M_2$ and
  $\partial^{in} M_1$ are corresponding to $T_2$.
\end{enumerate}
Under these notations, $N_1= M_2 \cup_{\varphi_2} M_1$ and $N_2 =M_1 \cup_{\varphi_1} M_2$.
We denote by $P_i: N_i \to S_i^1$ the quotient map of the fiber quotient space of $\cF_i$.

We can define an $m$-covering map $\pi: S_2^1 \to S_1^1$ ($m\in \mathbb{N}$) as follows.
For every  $z_2 \in S_2^1$, since $S_2^1$  can be regarded as the quotient space of the circle fibration $F_2$ on $\partial^{out} M_2$,
we can regard $z_2$ as a fiber of $F_2$. Also notice that $\partial^{out}M_2$ is embedded into $N_1$,
then the fiber $z_2$ is in some punctured disk fiber of $\cF_1$.
Therefore, the quotient map $P_1: N_1 \to S_1^1$ naturally induces a map $\pi: S_2^1 \to S_1^1$. One can easily check that
$\pi$ is an $m$-covering map.

We claim that $\varphi_2^1 \circ \pi=\pi \circ \varphi_2^2 $ which is the key observation for the proof.
Now let's check this claim. For every point $x_i \in N_i$ ($i=1,2$),
we denote by $\langle x_i \rangle_i \in S_i^1$ the fiber of $\cF_i$ where $x_i$ stays at. Let $x_2$ be a point in $\partial^{out} M_2 \subset N_2$
and $x_1$ be a point in $\partial^{out} M_1 \subset N_1$ such that $\langle x_1 \rangle_1 =\pi(\langle x_2 \rangle_2)$.
Then one can easily  show that
$P_1 \circ \varphi^1 (x_1) = \pi \circ P_2 \circ \varphi^2 (x_2)$ by following the definitions of $P_i$, $\varphi^i$ and $\pi$ ($i=1,2$).
Note
$P_1 \circ \varphi^1 (x_1) = \varphi_2^1 (\langle x_1 \rangle_1)$ and
$\pi \circ P_2 \circ \varphi^2 (x_2)= \pi\circ \varphi_2^2 (\langle x_2 \rangle_2)$.
By the equalities above, we have $\varphi_2^1 \circ \pi (\langle x_2 \rangle_2)=\pi \circ \varphi_2^2 (\langle x_2 \rangle_2)$ for
every $\langle x_2 \rangle_2 \in S_2^1$.

Since $\varphi_2^1$ is affine, we can endow a suitable metric on $S_1^1$ such that $S_1^1 =\{z\mid |z|=1, z\in \mathbb{C}\}$ and $\varphi_2^1 =z^n$
for some $n\in \mathbb{N}$ ($n\geq 2$). Since  $\pi: S_2^1 \to S_1^1$ is an  $m$-covering map,
we also can endow a metric on $S_2^1$ such that  $S_2^1 =\{z\mid |z|=1, z\in \mathbb{C}\}$ and $\pi (z)=z^m$ for every
$z\in S_2^1$. Moreover, by the fact that $\pi \circ \varphi_2^2 = \varphi_2^1 \circ \pi$, we have
$\varphi_2^2= z^n: S_2^1 \to S_2^1$.
\end{proof}

\begin{proposition} \label{p.oneside}
Let $T$ be an incompressible torus on a closed $3$-manifold $M$.  We denote $N$ by
the path closure of $M-T$, so that $\partial N$ is the union of
$T^{out}$ and $T^{in}$.
 Then up to isotopic leaf-conjugacy,
there exists at most one affine Hirsch foliation $\cH$ so that
$\cH$ is transverse to $T$ and
 $\cH \mid_{N}$
is a punctured disk fibration so that each fiber of $\cH \mid_{N}$ intersects  $T^{out}$ at one
connected component.
\end{proposition}

\begin{proof}
We assume that $\cH_1$ and $\cH_2$ are two affine Hirsch foliations on $M$ which satisfy the conditions in the proposition.
Let $\cF_1$ and $\cF_2$ be the punctured disk fibrations induced by $\cH_1$ and $\cH_2$ on $N$ respectively.
Suppose $\varphi: T^{out} \to T^{in}$ is the  gluing map so that $M=N\setminus x\sim \varphi (x), x \in T^{out}$.

$F_i^{out}=\cF_i \cap T^{out}$ ($i=1,2$) is a $S^1$-fibration on $T^{out}$. Similarly, $F_i^{in}=\cF_i \cap T^{in}$
is a $S^1$-fibration on $T^{in}$. We denote a fiber of $F_1^{out}$ (\emph{resp.} $F_1^{in}$, $F_2^{out}$, $F_2^{in}$) by
$m_1$ (\emph{resp.}  $m_2$, $c_1$, $c_2$). Then, up to isotopy, $\varphi(m_1)= m_2$ and $\varphi(c_1)=c_2$.
By Lemma \ref{l.fixonT}, $c_1$ is isotopic to $m_1$ in $T^{out}$
and $c_2$ is isotopic to $m_2$ in $T^{in}$.
Then, we can suppose that $\cH_1 \cap T= \cH_2 \cap T$, denoted by $F$. Here $F$ is a circle fibration on
$T$.

Since each of $\cH_1$ and $\cH_2$ is an affine Hirsch foliation, by Lemma \ref{l.affT},
the projective holonomy maps  $\varphi_2^1:S^1 \to S^1$ of $\cH_1$ and $\varphi_2^2 : S^1 \to S^1$ of $\cH_2$ relative to $T$
are conjugated by an orientation preserving homeomorphism $g: S^1 \to S^1$,
 \emph{i.e.} $\varphi_2^2 \circ g = g \circ \varphi_2^1$.

Recall that  $P:N \to S^1$  is a natural quotient map where $S^1$ is the fiber quotient space of $F$.
One can lift $g$ to a homeomorphism $G_T: T \to T$ so that,
\begin{itemize}
  \item $G_T$ is isotopic to the identity map on
$T$;
  \item $P\circ G_T =g \circ P$.
\end{itemize}
Since $G_T$ is isotopic to the identity map on
$T$, we can extend $G_T$  to a homeomorphism $G:M\to M$ which is isotopic to the identity map
on $M$. Assume that $\cH_1' = G(\cH_1)$ is also an affine Hirsch foliation on $N$.
Let $\cF_1'$ be the punctured disk fibrations induced by $\cH_1'$ on $N$.
By $P\circ G_T =g \circ P$ and $\varphi_2^2 \circ g = g \circ \varphi_2^1$,
one can quickly check that  the boundaries of $\cF_1'$ and $\cF_2$ are coherent, i.e.
 for every fiber $\Sigma_1 \subset \cF_1'$, there exists a fiber $\Sigma_2$ so that $\partial \Sigma_1 =\partial \Sigma_2$.
 Then by Lemma \ref{l.fixinterior}, one can build a homeomorphism $\phi: N\to N$ so that,
 \begin{itemize}
   \item $\phi$ is isotopic to the identity map on $N$ relative to $\partial N$;
   \item $\phi (\cF_1')=\cF_2$.
 \end{itemize}
 $\phi$ can automatically induce  a homeomorphism $\Phi$ on $M$ so that,
 \begin{itemize}
   \item $\Phi (x)=\phi (x)$ for every $x$ in the interior of $N$;
   \item $\Phi$ is isotopic to the identity map on $M$;
   \item $\Phi (\cH_1')= \cH_2$.
 \end{itemize}
 In summary, $\Phi\circ G$ is a homeomorphism on $M$ so that,
 \begin{itemize}
 \item $\Phi\circ G$ is isotopic to the identity map on $M$;
 \item $\Phi\circ G(\cH_1) =\cH_2$.
 \end{itemize}

\end{proof}

Now we can finish the proof of Theorem \ref{t.maint}, \emph{i.e.} up to isotopic leaf-conjugacy, a closed orientable $3$-manifold
admits at most two affine Hirsch foliations.

\begin{proof} [Proof of Theorem \ref{t.maint}]
Let $\cH$ be an affine Hirsch foliation
and $T$ be an incompressible torus in $M$.
By Lemma \ref{l.transT},  we can suppose that
$\cH$ is transverse to $T$.
We denote the path closure of $M-T$ by $N$
and the boundary of $N$ by the union of $T^{out}$
and $T^{in}$.
Then, $F=\cH \mid_N$ is a punctured disk fibration on $N$.
$F$ has the following two possibilities:
\begin{enumerate}
  \item each leaf of $F$ intersects to  $T^{out}$ at one
connected component;
  \item each leaf of $F$ intersects to $T^{in}$ at one
connected component.
\end{enumerate}
In every case, by  Proposition \ref{p.oneside}, up to isotopic leaf-conjugacy,
there exists at most  one affine Hirsch foliation. The conclusion of the theorem follows.
 \end{proof}


\section{Hirsch manifolds and exchangeably braided links}
\label{s.twohirsch}
In this section, we will focus on the study of Hirsch manifolds, \emph{i.e.} the closed $3$-manifolds which
admit two non-isotopically leaf-conjugate affine Hirsch foliations. First, we will introduce or recall some notations
(see Figure \ref{f.notationsection4} as an illustration\footnote{In the case of the figure,
  $c_1=c_1^1=l_1$. To avoid misunderstanding, we should point it out that generally, we can think $c_1=c_1^1$ but $l_1$ may not be isotopic to $c_1$.})
which are useful below:
\begin{itemize}
  \item $\cH_1$:  an affine Hirsch foliation transverse to $T$ in $M$;
  \item $N, T^{out}, T^{in}, \varphi$: $M=N\setminus x\sim \varphi(x)$, $\partial N =T^{out} \cup T^{in}$ and $\varphi: T^{out} \to T^{in}$ is the gluing homeomorphism;
  \item $m_1, l_1; m_2,l_2$: $\cH_1$ induces oriented simple closed curves $m_1, l_1$ in $T^{out}$ and $m_2, l_2$ in $T^{in}$
  which are defined at the beginning of Section \ref{s.proof}.
  \item $c_2$: $p_2 m_2 + q_2 l_2$ ($q_2> 0$), an oriented simple closed curve in $T^{in}$;
  \item $c_1$: $p_1 m_1 + q_1 l_1$, an oriented simple closed curve in $T^{out}$;
  \item $c_1^1, \dots, c_1^s$: $s$  pairwise disjoint oriented simple closed curves
which are parallel to $c_1$ in $T^{out}$;
  \item $\Sigma_2$: an oriented punctured disk in $N$ so that $\partial \Sigma_2$ is the union of  $c_1^1, \dots, c_1^s$
and $c_2$;
  \item $F_2$: an oriented punctured disk fibration on $N$ with a fiber $\Sigma_2$;
  \item $\varphi:T^{out} \to T^{in}$: $\varphi(m_1)=m_2$ and $\varphi(l_1)=l_2 + km_2$.
\end{itemize}

 \begin{figure}[htp]
\begin{center}
  \includegraphics[totalheight=7.5cm]{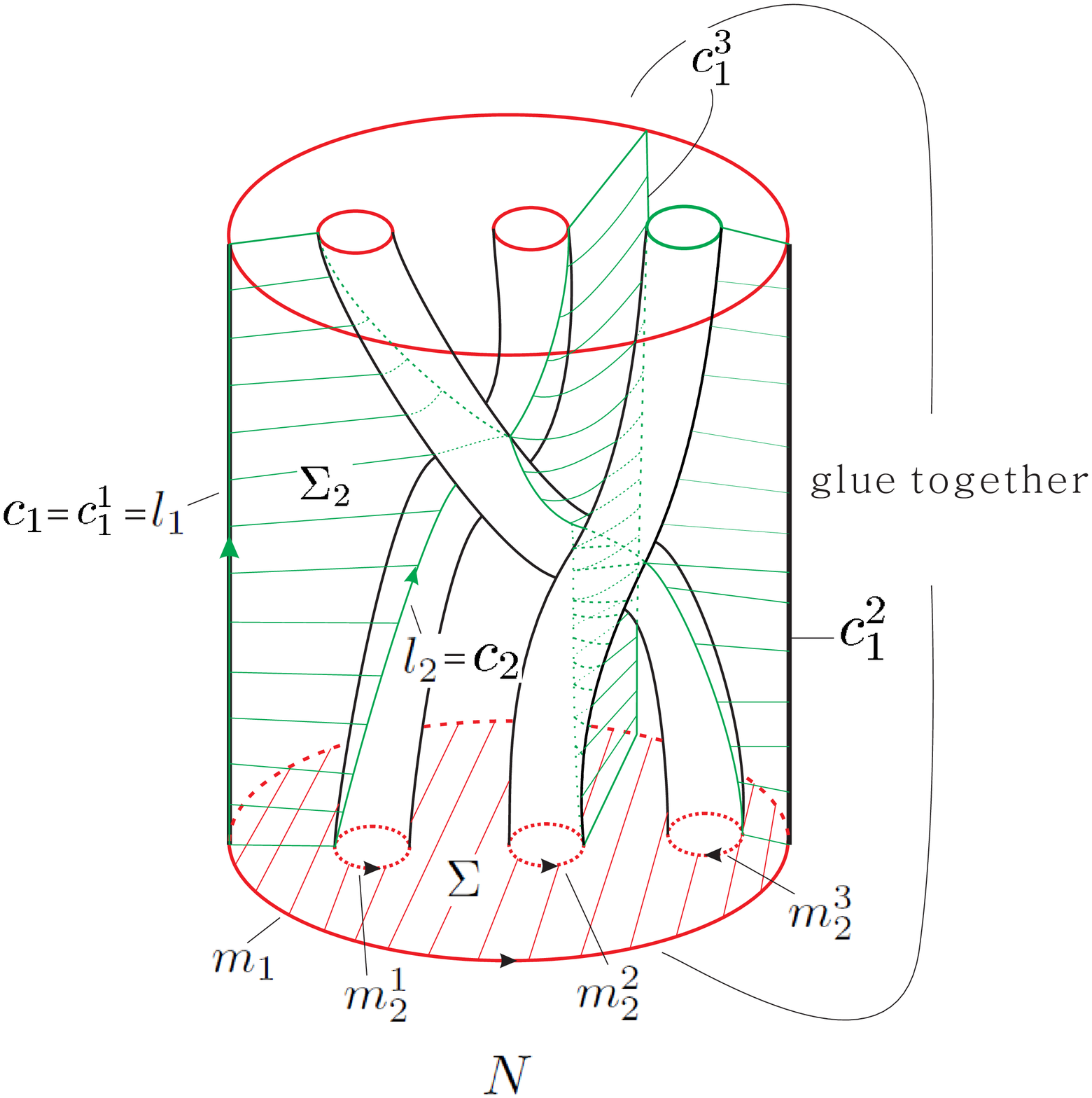}
  \caption{some notations in the case $N$ associated to the braid $\sigma_1\sigma_2^{-1}$}
  \label{f.notationsection4}
  \end{center}
\end{figure}

\subsection{Homology and Hirsch manifolds}
In this subsection, for every two nonzero integers $m$ and $n$, we will use $[m,n]$ to
represent their greatest positive common divisor.

\begin{lemma}\label{l.homologyF2}
Suppose $F_2$ also induces an affine Hirsch foliation $\cH_2$ on $M$ under $\varphi$.
Then, $p_1 =\frac{k}{[n^2 -1, k]}$, $q_1 =\frac{n^2 -1}{[n^2 -1, k]}$,
$p_2 =n^2 p_1$, $q_2 =q_1$ and $s=n$.
\end{lemma}

\begin{proof}
Since $\partial (\Sigma_2) =c_2 \cup c_1^1 \cup \dots \cup c_1^s$,
$[c_2]=s [c_1]$ in $H_1 (N)$.
Equivalently, $p_2 [m_2]+ q_2 [l_2]= sp_1 [m_1] +s q_1 [l_1]$.
Recall that $[m_1] =n [m_2]$ and $[l_2]= n[l_1]$, then,
$(snp_1 -p_2)[m_2] + (nq_2 -sq_1)[l_1]=0$. Recall that $H_1 (N)=\langle [m_2], [l_1] \rangle \cong \mathbb{Z}\oplus\mathbb{Z}$.
Then,

    \begin{equation}\label{r.homologyF2-1}
    p_2=snp_1, nq_2 =s q_1.
    \end{equation}

By filling a solid torus to  $N$ along $T^{out}$, we obtain a new compact $3$-manifold  $V$ so that $c_1$ bounds a disk in $V$, then
$V$ is homeomorphic to a solid torus. Following the gluing surgery, we have
\begin{eqnarray*}
H_1 (V) &=& \langle [m_2], [l_1]\mid p_1 [m_1]+ q_1 [l_1]=0 \rangle \nonumber\\
          &=& \langle [m_2], [l_1]\mid np_1 [m_2]+ q_1[l_1]=0 \rangle \\
          &\cong& \mathbb{Z}
\end{eqnarray*}

Then we have,
\begin{equation}\label{r.homologyF2-2}
   np_1  \hbox{ and }  q_1  \hbox { are coprime}.
    \end{equation}

Denote $\varphi_{\star}: H_1 (T^{out}) \to H_1 (T^{in})$ to be the homomorphism induced by $\varphi: T^{out} \to T^{in}$.
Notice that $F_2$ also induces an affine Hirsch foliation $\cH_2$ on $M$. Then, on the one hand,
\begin{eqnarray*}
    \varphi_{\star} ([c_1]) &=&[c_2]\nonumber\\
                            &=&p_2 [m_2] +q_2 [l_2];
 \end{eqnarray*}
 on the other hand,
 \begin{eqnarray*}
    \varphi_{\star}([c_1]) &=& \varphi_{\star} (p_1 [m_1]+q_1 [l_1])\nonumber\\
                           &=&p_1 [m_2] +q_1 (k[m_2] + [l_2])\nonumber\\
                           &=&(p_1 +q_1 k)[m_2] + q_1 [l_2].
\end{eqnarray*}

Therefore,
\begin{equation}\label{r.homologyF2-3}
   p_1 + q_1 k= p_2  \hbox{ and }  q_1 =q_2.
    \end{equation}

Now, the lemma is a direct consequence of (\ref{r.homologyF2-1}), (\ref{r.homologyF2-2}),
(\ref{r.homologyF2-3}) and the fact that $p_i$ and $q_i$ ($i=1,2$) are coprime.

\end{proof}

We can use the strand number of a braid which
build a (affine) Hirsch foliation on $M$ to be an invariant of $M$.
The strand number of $M$ is well defined. Let us explain a little bit more.
Maybe there are two braids to build the same Hirsch manifold $M$, by Lemma \ref{l.stdnber}, Definition \ref{d.strdnumberF} and Lemma \ref{l.homologyF2}, the strand numbers
of the braids are the same.

\begin{definition}\label{d.strandnumber}
Let $M$ be a Hirsch manifold.  The strand number of a braid $b$ which builds a (affine) Hirsch foliation on
$M$ is called the \emph{strand number} of $M$.
\end{definition}

\begin{proposition}
\label{p.nonisotopy}
Let $\cH_1$ and $\cH_2$ be two affine Hirsch foliations defined as above on a Hirsch manifold $M$.
Then $\cH_1$ and $\cH_2$ are not isotopically leaf-conjugate.
\end{proposition}

\begin{proof}
Otherwise, we
assume that there exists a homeomorphism $h:M\to M$ which maps every leaf of $\cH_1$ to a leaf of $\cH_2$ and
is isotopic to the identity map on $M$.
One can check that every leaf on $\cH_1$ is homeomorphic to either a sphere minus a Cantor set or a torus minus a Cantor set.
We choose a leaf $\ell_1$ on $\cH_1$ which is homeomorphic to a sphere minus a Cantor set.
We denote $f(\ell_1)$ by $\ell_2$ which is a leaf on $\cH_2$.

Let $Q: N\to M$ be the natural quotient map.
 By the construction of $\cH_1$, without loss of generality, we can assume that $b_1 = Q(m_1)=Q(m_2)$  is an oriented simple closed curve on $\ell_1$.
  $b_1$ is homotopically nontrivial in
$M$ because of the compressibility of $T$ in $M$.
Since $h$ is isotopic to the identity map on $M$, $b_2=h(b_1) \subset \ell_2$ is also homotopically nontrivial  in $M$.
By the construction of $\cH_2$, $b_2$ is homotopic to $\lambda c$ in $\ell_2$ for some  nonzero integer $\lambda$.
Here $c=Q(c_1)=Q(c_2)$ is the simple closed curve on $T$.
We choose an oriented closed curve $c_{\lambda}$ in $T$ which is homotopic to $\lambda c$ in $T$.
Then $b_1$ and $c_{\lambda}$ are homotopic in $M$.
This means that there exists an immersion map $F: A=S^1 \times [0,1] \to M$ and an orientation on $A$ so that,
\begin{itemize}
  \item $F(S^1 \times \{0\})=l_1$ and $F(S^1 \times \{1\})=c_{\lambda}$ where $S^1 \times \{0\}$ and $S^1 \times \{1\}$ are oriented which are coherent to the orientation of $A$;
  \item $F(\emph{int}(A))$ is transverse to $T$ where $\emph{int}(A)$ is the interior of $A$.
\end{itemize}
Moreover, under some pertubation of $F$ close to $\partial A$ if it is necessary, we can assume that
there exists a neighborhood of $\partial A$, denoted by $N(A)$, satisfies  that $F^{-1}(T)\cap N(\partial A)= \partial A$.
Then, $F^{-1}(T)\cap \emph{int}(A)$ is the union of finitely many, pairwise disjoint oriented simple closed curves
$s_0, s_1 \dots, s_m$ where $s_0 =S^1 \times \{0\}$ and $s_m = S^1 \times \{1\}$.
Here the orientation of $s_i$ ($i\in \{0,1, \dots,m\}$) is coherent to the orientation of $s_0$ in $A$.
We can assume that $m$ is minimal in the following sense:
let $F: A\to M$ be an immersion which satisfies the conditions above,
then   $F^{-1}(T)\cap \emph{int}(A)$ contains at least $m$ connected components.
If some  $s_i$ is inessential in $A$, then $s_i$ bounds a disk $D_i$ in $A$. This means that
$F(s_i)$ is homotopically trivial in $M$. Since $F(s_i)\subset T$ and $T$ is incompressible in $M$,
$F(s_i)$ is homotopically trivial in $T$. Then by some standard surgery, one can build another
$F': A\to M$ which satisfies the conditions above and whose intersection circle number is less than $m$. It contradicts  the assumption for
$m$. Therefore, from now on, we can suppose that each $s_i$ is essential in $A$.
Since $A$ is an annulus, $s_0, s_1 \dots, s_m$  are pairwise parallel in $A$.

Without loss of generality, we can assume that  the union of $s_0, s_1 \dots, s_m$ cuts
$A$ to $m$ open annuli $A_1, \dots, A_m$ so that,
\begin{itemize}
  \item $\partial A_i = s_{i-1}\cup s_i$ ($i\in\{1,\dots,m\}$);
  \item $A_i \cap F^{-1}(T)=\emptyset$.
\end{itemize}
Therefore, $F(s_{i-1})$ and $F(s_i)$ are homotopic in $N$ for every  $i\in\{1,\dots,m\}$.
We choose a very small tubular neighborhood of $s_0$ in $A$. Then $F(N(s_0))$ belongs to
one of the two sides of $T$ in $M$. The two cases for the position of $F(N(s_0))$  and the
relations above
induce two  kinds of ``homotopy chain relation". We denote $Q^{-1} (c_{\lambda})$ by
$c_{\lambda}^1 \cup c_{\lambda}^2$ where $c_{\lambda}^1 \subset T^{in}$ and
$c_{\lambda}^2 \subset T^{out}$.
In both cases, we can assume that there exist $2m$ oriented  closed curves $s_1^1, s_2^1, \dots, s_m^1$ in $T^{out}$ and
$s_0^2, s_1^2, \dots, s_{m-1}^2$ in $T^{in}$ so that:
\begin{itemize}
  \item $Q(s_i^1)=Q(s_i^2)=F(s_i)$ and $s_i^2 = \varphi(s_i^1)$ ($i\in \{1,\dots,m-1\}$);
  \item $s_{i-1}^2$ and $s_i^1$ are homotopic in $N$ ($i\in \{1,\dots,m\}$).
\end{itemize}
In  one case, $s_0^2 =c_{\lambda}^2$ and $s_m^1=m_1$;
in the other case,  $s_0^2 =m_2$ and $s_m^1=c_{\lambda}^1$.

We will get contradictions in both cases by using homology theory. For every oriented closed curve
$\alpha$ in $N$,
we will use $[\alpha]$ to represent the corresponding homological element in $H_1(N)$.
Recall that
$H_1 (N)\cong \mathbb{Z}\oplus \mathbb{Z}$ which is generated by $[l_1]$ and $[m_2]$ (Lemma \ref{l.homologyN}).
Moreover, $[l_2]=n[l_1]$ and $[m_1]=n[m_2]$.  These facts will be used several times in the following.

In the first case,
on the one hand, since $s_0^2=m_2$ and $s_1^1$ are homotopic in $N$, $[m_2]=[s_1^1]$ in $H_1 (N)$;
on the other hand, since $s_1^1$ is an oriented closed curve in $T^{out}$, then $[s_1^1]=r[m_1]+t[l_1]$ for two integers $r$
and $t$.
These two sides imply that $[s_1^1]=nr[m_2]+t[l_1]=[m_2]$  in $H_1 (N)$.
Notice that $n>1$, then the equality is impossible. Therefore, we obtain a contradiction.

Now we discuss the second case.
Since $s_{i-1}^2$ and $s_i^1$ are homotopic in $N$ ($i\in \{1,\dots,m\}$),
$[s_{i-1}^2]=[s_i^1]$.
In particular, $[s_{m-1}^2]=[s_m^1]=[m_1]$.
Since $s_{m-1}^2$ is an oriented closed curve in $T^{in}$, then $[s_{m-1}^2]=r_{m-1}[m_2]+t_{m-1}[l_2]$ for two integers $r_{m-1}$
and $t_{m-1}$. We have $[s_{m-1}^2]=r_{m-1}[m_2]+n t_{m-1}[l_1]=n[m_2]$. Therefore, $r_{m-1}=n$ and $t_{m-1}=0$.
This implies that $s_{m-1}^2$ and $n m_2$  are homotopic in $T^{in}$. Notice that $s_{m-1}^1=\varphi^{-1}(s_{m-1}^2)$
and $\varphi(m_1)=m_2$, then $s_{m-1}^1$ and $nm_1$ are homotopic in $T^{out}$.
By some similar arguments, we have that $s_i^1$ and $n^{m-i} m_1$ are homotopic in $N$ for every $i\in \{1,\dots, m-1\}$.
Also notice that $s_1^1$, $s_0^2$ and $c_{\lambda}$ are pairwise homotopic, then $n^{m-1} m_1$ and $c_{\lambda}$
are homotopic in $N$. This implies that $n^{m-1} [m_1]=[c_{\lambda}] =n^{m-1+1} [m_2]=n^m [m_2]$.
By Lemma \ref{l.homologyF2}, $[c_{\lambda}]=\lambda [c_2]=\lambda (n^2 \frac{k}{[n^2 -1, k]} [m_2]+
\frac{n^2 -1}{[n^2 -1, k]}[l_2])$.  Since $\frac{n^2 -1}{[n^2 -1, k]}[l_2]= \frac{n(n^2 -1)}{[n^2 -1, k]}[l_1]$ is nonzero,
then $[c_{\lambda}] \neq n^m [m_2]$. We obtain a contradiction.

Then, the proposition is followed.

\end{proof}

\begin{remark}\label{r.HirvsaffHir}

By Definition \ref{d.hirsch} and Definition \ref{d.affinehirsch}, we can see that
for a given $3$-manifold $M$,
 \begin{itemize}
   \item on the one hand, every Hirsch foliation can be obtained from a unique affine Hirsch foliation by replacing
the projective holonomy map $\varphi_2=z^n$ on $S^1$ by another degree $n$ endomorphism
  $\varphi_2'$ on $S^1$;
   \item on the other hand, for every affine Hirsch foliation and every degree $n$ endomorphism   $\varphi_2'$ on $S^1$,
   one can build a Hirsch foliation with the projective holonomy map $\varphi_2'$.
 \end{itemize}

Moreover, by Proposition \ref{p.nonisotopy} and Theorem \ref{t.maint}, one can classify all of the affine Hirsch foliations
on a given $3$-manifold $M$.

Therefore, our results conclude the question classifying all Hirsch foliations to a  classical field on one dimensional dynamical system:
classifying degree $n$ ($n\geq 2$) endomorphisms on $S^1$ up to conjugacy.
\end{remark}

\subsection{DEBL Hirsch manifolds}
\label{sb.His}

To be more convenient to understand the materials in this subsection, we suggest the reader to look at Figure \ref{f.DEBLHirsch}.

 \begin{figure}[htp]
\begin{center}
  \includegraphics[totalheight=8.35cm]{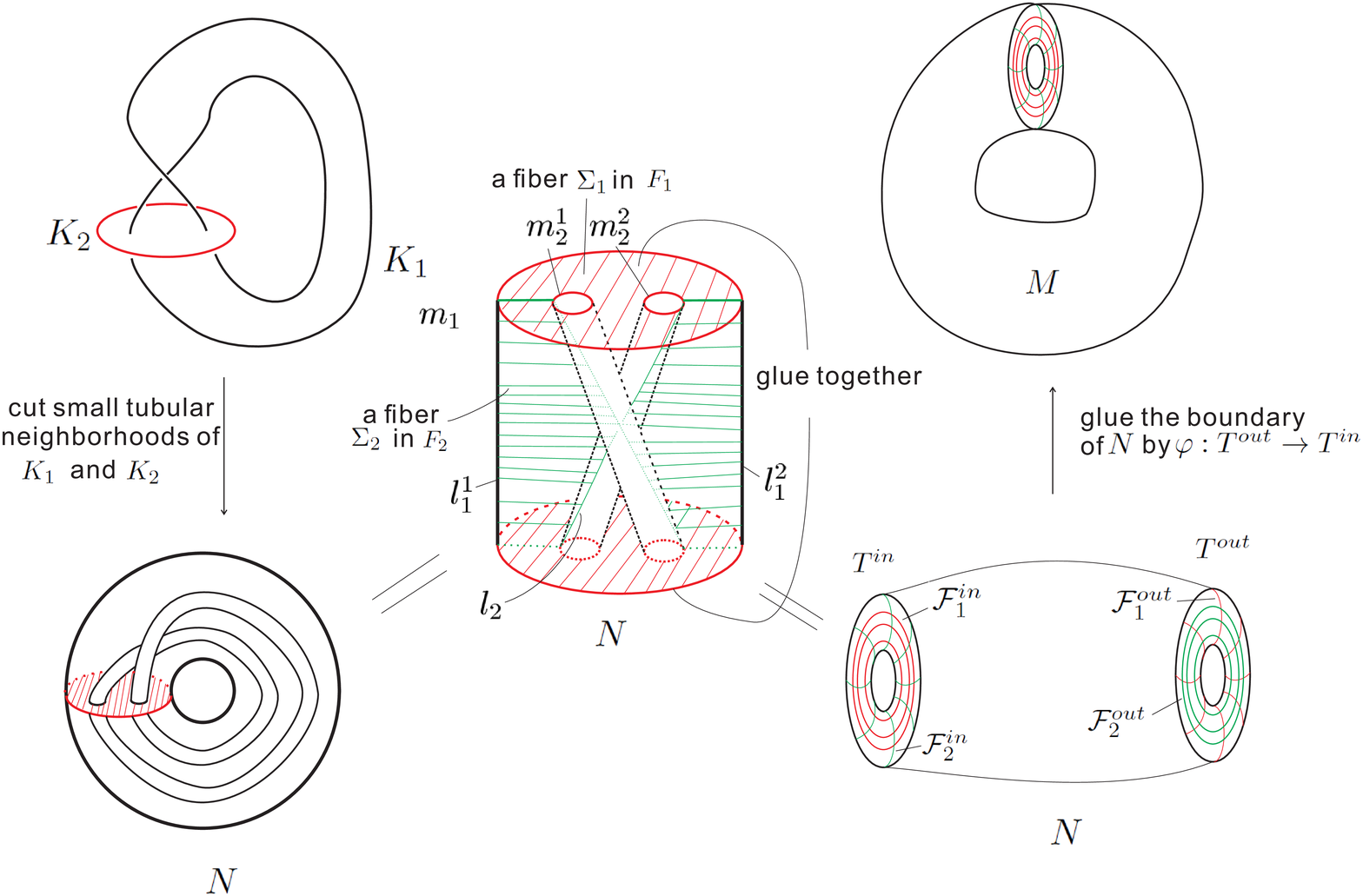}
  \caption{An example of DEBL Hirsch manifolds}\label{f.DEBLHirsch}
  \end{center}
\end{figure}

Let $L= K_1 \cup K_2$ be an exchangeably braided link in $S^3$.
We choose two disjoint small open tubular neighborhoods $V_1$ and $V_2$ of $K_1$ and $K_2$ respectively.
$N$ is defined to be $S^3 - V_1 \cup V_2$.
$\partial N=T^{out} \cup T^{in}$ so that $T^{out}= \partial \overline{V_1}$ and
                                          $T^{in}= \partial \overline{V_2}$.
The linking number of $K_1$ and $K_2$ is defined to be $n$.
$K_1$ is a closed $n$-braid $\widetilde{b_1}$ relative to $K_2$
and $K_2$ is a closed $n$-braid  $\widetilde{b_2}$ relative to $K_1$.

Up to isotopy, there is a unique way to choose a simple closed curve $m_1$ in $T^{out}$ and
$n$ simple closed curves $m_2^1, \dots, m_2^n$ in $T^{in}$ so that,
\begin{itemize}
  \item $m_2^i$ ($i=1,\dots, n$) bounds a disk in $\overline{V_2}$ and $m_1$ is isotopic to $K_1$ in $\overline{V_1}$;
  \item $m_2^1, \dots, m_2^n$ and $m_1$ bound an $n$-punctured disk $\Sigma_1$ in $N$.
\end{itemize}

Similarly,
up to isotopy, there is a unique way to choose a simple closed curve $l_2$ in $T^{in}$ and
$n$ simple closed curves $l_1^1, \dots, l_1^n$ in $T^{out}$ so that,
\begin{itemize}
  \item $l_1^i$ ($i=1,\dots,n$) bounds a disk in $\overline{V_1}$ and $l_2$ is isotopic to $K_2$ in $\overline{V_2}$;
  \item $l_1^1, \dots, l_1^n$ and $l_2$ bound an $n$-punctured disk $\Sigma_2$ in $N$.
\end{itemize}

Moreover, up to isotopy, we can extend $\Sigma_i$ ($i=1,2$) to an $n$-punctured disk fibration $F_i$ on $N$
so that,
\begin{itemize}
  \item $\cF_i^{out}=F_i \cap T^{out}$ and $\cF_i^{in}=F_i \cap T^{in}$ which are two $S^1$-fibrations on $T^{out}$
and $T^{in}$;
  \item $\cF_1^{out}$ and $\cF_2^{out}$ transversely intersects every where on $T^{out}$;
  \item $\cF_1^{in}$ and $\cF_2^{in}$ transversely intersects every where on $T^{in}$.
\end{itemize}

We can suppose that the intersection number of $m_1$ and $l_1^i$ is $1$ and
the intersection number of $m_2^i$ and $l_2$ (for every $i\in \{1,\dots, n\}$) is $1$.
Similar to the beginning of Section \ref{s.prel} and Section \ref{s.proof},
we would like to provide some orientations on these objects:
\begin{itemize}
  \item give $N$ an orientation;
  \item $T^{in}$ and $T^{out}$ are oriented by the  orientation of $N$;
  \item give the leaves of $F_1$ and $F_2$ some orientations continuously so that:
  \begin{enumerate}
   \item every fiber of $\cF_i^{out}$ and $\cF_i^{in}$ ($i=1,2$) is oriented which is induced
  by the orientation of $F_1$ and $F_2$;
  \item the orientation of $T^{out}$ is coherent to $(m_1, l_1^1)$ and
       the orientation of $T^{in}$ is coherent to $(m_2^1, l_2)$.
\end{enumerate}
\end{itemize}

Then, one can build an orientation preserving homeomorphism
$\varphi: T^{out} \to T^{in}$ so that,
\begin{itemize}
  \item $\varphi(m_1)= m_2^1$ and $\varphi(l_1^1)=l_2$;
  \item $\varphi$ maps every fiber of $\cF_1^{out}$ to a fiber of $\cF_1^{in}$;
  \item $\varphi$ maps every fiber of $\cF_2^{out}$ to a fiber of $\cF_2^{in}$;
  \item $F_1$ and $F_2$ induce two affine Hirsch foliations $\cH_1$ and $\cH_2$ on $M$ under $\varphi$
       where $M= N\setminus x\sim \varphi(x), x \in T^{out}$.
\end{itemize}

Notice that to ensure the glued manifold $M$ can admit two Hirsch foliations induced by $F_1$ and $F_2$,
up to isotopy, we can suppose that $\varphi(m_1)= m_2^1$ and $\varphi(l_1^1)=l_2$. This implies that under this restriction,
the glued manifold $M$ is unique up to homeomorphism. Therefore, we can say that
an exchangeably braided link  decide a unique Hirsch manifold.
Every Hirsch manifold built in this way is called a \emph{Hirsch manifold derived from an exchangeably braided link}
(abbreviated as a \emph{DEBL Hirsch manifold}).

By the second item of Corollary \ref{c.finiteness}, we have the following consequence.

\begin{corollary}\label{c.finiDEBL}
For every $n\in \mathbb{N}$, there are only finitely many DEBL Hirsch manifolds with strand number $n$.
\end{corollary}

\subsection{A virtual property of Hirsch manifolds}
\label{s.virp}


Let $M$ be a Hirsch manifold. By the definition of Hirsch manifold and Lemma \ref{l.transT}, there exist two affine Hirsch foliations $\cH_1$ and $\cH_2$ on $M$ and a JSJ torus $T$ in $M$
which satisfy the propositions in Lemma \ref{l.transT}. Let $N$ be the path closure of $M-T$, then
\begin{itemize}
  \item $N$ admits two $n$-punctured disk fibrations
   $F_1$ and $F_2$ and we can parameterize some subsets of $N$ as Section \ref{s.prel};
  \item the two Hirsch foliations $\cH_1$ and $\cH_2$ can be induced by $F_1$ and $F_2$ and a gluing map $\varphi: T^{out}\to T^{in}$ respectively.
\end{itemize}

$H_1 (N)\cong \mathbb{Z}\oplus \mathbb{Z}$ which is generated by $[m_2]$ and $[l_1]$.
We denote the abelization homomorphism from $\pi_1 (N)$ to $H_1 (N)$ by  $\psi_1$ and denote the quotient
homomorphism from $H_1 (N)$ to $\mathbb{Z}_{q_2}$ so that $\psi_2 ([m_2])=0$ and $\psi_2 ([l_1])=1$ by $\psi_2$.
The kernel of $\psi_N= \psi_2 \circ \psi_1 : \pi_1 (N) \to \mathbb{Z}_{q_2}$, $G= \ker (\psi_N) $,  is a normal subgroup of
$\pi_1 (N)$. Here, the definitions and properties of $m_i$, $l_i$ ($i=1,2$) and $q_2$
can be found in the beginning of Section \ref{s.twohirsch} and Lemma \ref{l.homologyF2}.

As a subgroup of $\pi_1(N)$,
$G$ induces an $q_2$-covering space of $N$ by a covering map $P: \widetilde{N} \to N$.
We collect some useful properties as the following. One can prove the proposition by some routine checks.
We omit the details here.

\begin{proposition} \label{p.P}
($i=1,2$)
    \begin{enumerate}
      \item $P^{-1} (F_i)=\widetilde{F_i}$  is an $n$-punctured disk fibration on $\widetilde{N}$;
      \item Let $\widetilde{\Sigma_1}$ be a connected component  of $P^{-1} (\Sigma_1)$, then
      $P:\widetilde{\Sigma_1} \to \Sigma_1$   is a homeomorphism so that $P(\widetilde{m_i})=m_i$;
      \item $P:\widetilde{l_1} \to l_1$  is a $q_2$-covering map;
      \item Let $\widetilde{\Sigma_2}$ be a connected component  of $P^{-1} (\Sigma_2)$, then
      $P:\widetilde{\Sigma_2} \to \Sigma_2$   is a homeomorphism so that $P(\widetilde{c_i})= c_i$;
      \item $\widetilde{c_i}$ intersects $\widetilde{m_i}$ at one point.
    \end{enumerate}
\end{proposition}

\begin{lemma}\label{l.lifttorus}
There is a homeomorphism $\widetilde{\varphi}: \widetilde{T}^{out} \to \widetilde{T}^{in}$ so that,
\begin{enumerate}
  \item $P\circ \widetilde{\varphi}= \varphi\circ P: \widetilde{T}^{out} \to T^{in}$;
  \item $\widetilde{\varphi}(\widetilde{m_1})= \widetilde{m_2}$ and $\widetilde{\varphi}(\widetilde{c_1})= \widetilde{c_2}$;
\end{enumerate}
\end{lemma}
\begin{proof}
By Proposition \ref{p.P}, $\widetilde{m_i} \cap \widetilde{c_i}$ ($i=1,2$) is one point, which can be defined by $\widetilde{x_i}$.
We denote $P(\widetilde{x_i})$ by $x_i$, then $\varphi(x_1)=x_2$.
$(\varphi \circ P)_{\star} (\pi_1 (\widetilde{T}^{out}, \widetilde{x_1}))
= \langle\varphi_{\star}([m_1]), \varphi_{\star}([c_1]) \rangle
= \langle [m_2], [c_2] \rangle \lhd \pi_1 (T^{in},x_2)$

$P_{\star} (\pi_1 (\widetilde{T}^{in}, \widetilde{x_2}))
=\langle [m_2], [c_2] \rangle \lhd \pi_1 (T^{in},x_2)$

Then by the classical homotopy lifting theorem, we can build a unique map $\widetilde{\varphi}: \widetilde{T}^{out} \to \widetilde{T}^{in}$
so that,
\begin{enumerate}
  \item $\widetilde{\varphi} (\widetilde{x_1}) =\widetilde{x_2}$, $\widetilde{\varphi}(\widetilde{m_1})= \widetilde{m_2}$;
  \item $P\circ \widetilde{\varphi}= \varphi\circ P: \widetilde{T}^{out} \to T^{in}$.
\end{enumerate}
Now one can automatically check that $\widetilde{\varphi}$ is a homeomorphism.
\end{proof}

We denote $\widetilde{N}\setminus y\sim \widetilde{\varphi}(y)$ ($y \in \widetilde{T}^{out}$) by $\widetilde{M}$.
$Q: \widetilde{N} \to \widetilde{M}$ and $q: N \to M$ are defined to be the corresponding quotient maps.

\begin{lemma}\label{l.liftM}
There is a unique map $\pi: \widetilde{M} \to M$ so that $\pi\circ Q(y)= q\circ P (y)$ for every $y \in \widetilde{N}$.
Furthermore, $\pi: \widetilde{M} \to M$ is a $q_2$-covering map.
\end{lemma}
\begin{proof}
For every $\widetilde{x}\in \widetilde{M}$, $\pi(\widetilde{x})$ can be defined as follows.
Since $Q:\widetilde{N}\to \widetilde{M}$ is surjective, there exists $\widetilde{y}\in \widetilde{N}$
so that $\widetilde{x} = Q (\widetilde{y})$.
Define $\pi (\widetilde{x})= q\circ P (\widetilde{y})$.
The first item of  Lemma \ref{l.lifttorus}
ensures that $\pi$ is well defined.
The uniqueness of $\pi$ is because of the fact that $\pi$ has no freedom in $\widetilde{M}- \widetilde{T}$ where $\widetilde{T}=\pi^{-1} (T)$.

Finally, since $P: \widetilde{N} \to N$ is a $q_2$-covering map, $\pi: \widetilde{M} \to M$ is also a $q_2$-covering map.
\end{proof}

\begin{lemma}\label{l.liftfoliation}
$\widetilde{M}$ is a Hirsch manifold which admits two affine Hirsch foliations $\widetilde{\cH_1}$ and
$\widetilde{\cH_2}$ so that $\widetilde{\cH_i}$ ($i=1,2$) is induced by $\cH_i$ under $\pi$, \emph{i.e.} $\pi$ maps each leaf of $\widetilde{\cH_i}$  to a leaf of $\cH_i$.
\end{lemma}
\begin{proof}
Assume that $\cF_i^{out}=F_i \cap T^{out}$ and $\cF_i^{in}=F_i \cap T^{in}$ which are two $S^1$-fibrations on $T^{out}$
and $T^{in}$ respectively.
Since $F_i$ induces $\cH_i$ on $M$ under the gluing homeomorphism
$\varphi: T^{out} \to T^{in}$, $\varphi$ maps every fiber of $\cF_i^{out}$ to a fiber of $\cF_i^{in}$.

Suppose $\widetilde{\cF_i}^{out}$ and $\widetilde{\cF_i}^{in}$ are the lifted fibrations of $\cF_i^{out}$ and
 $\cF_i^{in}$ on $\widetilde{T}^{out}$ and $\widetilde{T}^{in}$ under the covering map $P$
respectively.
$\widetilde{\varphi}: \widetilde{T}^{out} \to \widetilde{T}^{in}$ is the lifted map of
$\varphi: T^{out} \to T^{in}$,
\emph{i.e.} $P\circ \widetilde{\varphi}= \varphi\circ P: \widetilde{T}^{out} \to T^{in}$.
Therefore, $\widetilde{\varphi}$ maps every every fiber of $\widetilde{\cF_i}^{out}$ to a fiber of $\widetilde{\cF_i}^{in}$.
Then, $\widetilde{F_i}$ induces a Hirsch foliation $\widetilde{\cH_i}$ on $\widetilde{M}$.

To finish the proof, now we only need to check that $\widetilde{\cH_i}$ is an affine Hirsch foliation.
This actually is a consequence of the following facts:
\begin{itemize}
  \item $\widetilde{\varphi}$ is the lifted map of $\varphi$;
  \item $\cH_i$ is an affine Hirsch foliation;
  \item every expanding map on $S^1$ is topologically conjugate to an affine map on $S^1$
  with the same degree.
\end{itemize}
 \end{proof}

\begin{lemma}\label{l.MDEBL}
$\widetilde{M}$ is a DEBL Hirsch manifold.
\end{lemma}
\begin{proof}
We glue two solid tori $\widetilde{V_1}$ and $\widetilde{V_2}$ to $\widetilde{N}$ along its boundary $\widetilde{T}^{in} \cup \widetilde{T}^{out}$ respectively
by the gluing maps
$\phi_1: \partial \widetilde{V_1} \to \widetilde{T}^{in}$ and
 $\phi_2: \partial \widetilde{V_2} \to \widetilde{T}^{out}$
so that, $\widetilde{m_2}$ bounds a disk in $\widetilde{V_2}$ and $\widetilde{c_1}$ bounds a disk in $\widetilde{V_1}$.
Then, the glued manifold is homeomorphic to $S^3$.

Let $K_i$ ($i=1,2$)  be a simple closed curve in  $\widetilde{V_i}$  so that $\widetilde{V_i}$
is a tubular neighborhood of $K_i$.

Since $\widetilde{F_2}$ is a punctured disk fibration structure on $\widetilde{N}$
 and $\widetilde{m_2}$ bounds a disk in $\widetilde{V_2}$, the union of $\widetilde{V_2}$ and $\widetilde{N}$, named by $\widetilde{U_2}$,
is also homeomorphic to a solid torus. Obviously, $K_2$ is a closed braid in $\widetilde{U_2}$. Since
$S^3= \widetilde{V_2}\cup\widetilde{N} \cup \widetilde{V_2}$, automatically,
$K_2$ is a closed braid relative to $K_1$, \emph{i.e.} $K_2$ is a closed braid in $S^3 -K_1$.

Similarly, one can show that $K_1$ is a closed braid relative to $K_2$. Therefore, $L= K_1 \cup K_2$ is an exchangeably braided link.
Now one can automatically build the Hirsch manifold derived from $L$ and check that the Hirsch manifold is homeomorphic to $\widetilde{M}$.
\end{proof}

\begin{proof} [Proof of Theorem \ref{t.covering}]
The first part of Theorem \ref{t.covering} is a direct consequence of Lemma \ref{l.liftM} and Lemma \ref{l.MDEBL}.
Moreover, by Lemma \ref{l.homologyF2}, $q_2$ can be divided by $n^2-1$.
\end{proof}

\subsection{Finiteness of Hirsch manifolds with stand number $n$}

We will use the following theorem of Wang \cite{Wang}.

\begin{theorem}\label{t.wang}
Let $M$  be a closed irreducible $3$-manifold which is nonorientable or Seifert fibered or has a nontrivial torus decomposition (\emph{i.e.}
there is a JSJ torus).
 Then $M$  covers infinitely many nonhomeomorphic 3-manifolds if and only if M  is an orientable Seifert fiber space with nonzero Euler number.
\end{theorem}

\begin{proof} [proof of Proposition \ref{p.finiteMs}]
 On the one hand, By Proposition \ref{p.topchar},  an $n$-strand Hirsch manifold
is an irreducible orientable closed $3$-manifold with some JSJ tori.  By Theorem \ref{t.wang}, for an given DEBL Hirsch manifold $\widetilde{M}$,
there are only finitely many Hirsch manifolds with $\widetilde{M}$ as a finite covering space.

On the other hand,
Corollary \ref{c.finiteness} says that for a positive integer $n$, up to isotopy, there are only finitely many exchangeably braided links with strands
number $n$. Recall that an exchangeably braided link decide a DEBL Hirsch manifold.
Therefore, there are only finitely many DEBL Hirsch manifolds with strand number $n$.

Let $M$ be a Hirsch manifold with strand number $n$. Then by Theorem \ref{t.covering}, a finite covering space of $M$,
$\widetilde{M}$, is a DEBL Hirsch manifold with strand number $n$.
Combing the two sides above, up to homeomorphism, there are only   finitely many
  Hirsch manifolds with strand number $n$.
 \end{proof}

\section{Proof of Proposition \ref{p.notHirschmfd}}
\label{s.example}
In this section, we will construct an example to prove Proposition \ref{p.notHirschmfd},
i.e. there exists a $3$-manifold which admits an affine Hirsch foliation but is not a Hirsch manifold.
We will use the following inequality by Bennequin \cite{Be}:

\begin{lemma} [Bennequin inequality]
\label{l.Beinequalty}
Let $L$ be a non-separating link of $\mu$ components, presented by a closed braid with $l$
strands and $c_+$ ($c_-$) positive (negative) crossings. Then $g(L)$, the genus of $L$, is bounded
as follows:
 $$\frac{|c_+ -c_-| -l-\mu}{2} +1 \leq g(L)\leq \frac{|c_+ +c_-| -l-\mu}{2} +1.$$
\end{lemma}

\begin{proof}[Proof of Proposition \ref{p.notHirschmfd}]
Let $b=(\sigma_1 \sigma_2^{-1})^2$ be a $3$-strand braid.
 Now we can follow the beginning of Section \ref{s.prel} to build an affine Hirsch foliation
 $\cH$ on a closed $3$-manifold $M$. We briefly recall the construction here.
 \begin{itemize}
   \item $b$ also can be used to represent a diffeomorphism
  on a $3$-punctured disc $\Sigma$ and we denote the mapping torus of $(\Sigma, b)$ by $N$.
   \item $F= \{\Sigma \times \{\star\}\}$ provides a $3$-punctured disk fibration on  $N$,
   which provides $T^{in}$ and $T^{out}$ two $S^1$-fibration structures $\cF_1$ and $\cF_2$
respectively.
   \item After carefully choosing  orientations to the objects above,
   we can build an orientation-preserving homeomorphism $\varphi: T^{out} \to T^{in}$ which maps every fiber of $\cF_1$ to a fiber of $\cF_2$
  and preserves the corresponding orientations.
  \item Let $M =N\setminus x\sim \varphi(x), x\in T^{out} N$. Then  $F$ naturally induces a Hirsch foliation $\cH$ on $M$ by $\varphi$.
  If we choose $\varphi$ suitably, $\cF$ is an affine Hirsch foliation.
 \end{itemize}
 Now  we assume that $M$ is a Hirsch manifold.
 Following the arguments in subsection \ref{s.virp}, there exists some integer $p$ so that the braid $b^{q_2} \tau^p$ is an exchangeable braid where $\tau$ is a $3$-strand full twist braid. This means that
 the knot $K= \widetilde{b^{q_2} \tau^p}$, the closed braid of $b^{q_2} \tau^p$, is a trivial knot.
 In the following, we will show that the genus of $K$, $g(K)$, is nonzero.
 Then $K$ isn't  a trivial knot.  We obtain a contradiction. Then  $M$ isn't a Hirsch manifold.

Comparing with the notations in Lemma \ref{l.Beinequalty}, in our case, $L=K= \widetilde{b^{q_2} \tau^p}$,  $l=3$, $\mu=1$ and $|c_+ -c_-| =6 |p|$.
By  Lemma \ref{l.Beinequalty}, $g(K)\geq 3|p|-1$.
Therefore, if $g(K)=0$, then $p=0$.
In the case $p=0$, $K=\widetilde{b^{q_2}}$. By Lemma \ref{l.homologyF2}, $q_2$ is nonzero.
Actually, it is well known that in this case, $\widetilde{b^{q_2}}$
is a genus $1$ fiber knot (see, for instance, Rolfsen  \cite [Chapter 10] {Ro}). Therefore, $K$ isn't a trivial knot.
\end{proof}

\section*{Acknowledgments}
The author would like are grateful to S\'ebastien Alvarez for useful discussions.
The author is supported by the National Natural
Science Foundation of China (grant no. 11471248).

%
%
%
%

\end{document}